\documentclass[12pt,twoside,english,shortlabels]{amsart}

\usepackage{babel}
\usepackage{amstext}
\usepackage{amssymb}

\usepackage{amsmath} 
\usepackage{latexsym}
\usepackage{graphicx} 

\usepackage{mathptmx}

\usepackage{graphicx} 

\makeatletter

\usepackage{amsthm}
\usepackage{enumitem}		

\usepackage{mathrsfs} 
\providecommand{\noopsort}[1]{}

\numberwithin{equation}{subsection}

\theoremstyle{definition} 

 \newtheorem{definition}{Definition}[section]
 \newtheorem{remark}[definition]{Remark}
 \newtheorem{example}[definition]{Example}

\newtheorem*{notation}{Notations}

\theoremstyle{plain}      

 \newtheorem{proposition}[definition]{Proposition}
 \newtheorem{theorem}[definition]{Theorem}
 \newtheorem{corollary}[definition]{Corollary}
 \newtheorem{lemma}[definition]{Lemma}

\makeatletter
\newcommand*{\house}[1]{
  \mathord{
    \mathpalette\@house{#1}
  }
}
\newcommand*{\@house}[2]{
  \dimen@=\fontdimen8 %
      \ifx#1\scriptscriptstyle\scriptscriptfont
      \else\ifx#1\scriptstyle\scriptfont
      \else\textfont\fi\fi
      3 %
  \sbox0{%
    $#1%
      \vrule width\dimen@\relax
      \overline{%
        \kern2\dimen@
        \begingroup 
          #2%
        \endgroup
        \kern2\dimen@
      }
      \vrule width\dimen@\relax
      \mathsurround=1.5\dimen@ 
    $
  }
  \ht0=\dimexpr\ht0-\dimen@\relax
  \dp0=\dimexpr\dp0+2\dimen@\relax
  \vbox{
    \kern\dimen@ 
    \copy0 
  }
}

\def\dyg{{\rm dyg}}

\def\11{{\mathbf 1}}

\theoremstyle{remark}
\newtheorem{exampl}[subsubsection]{Example}

\def\bee{\begin{exampl}}
\def\eee{\end{exampl}}
\def\bn{\begin{notation}}
\def\en{\end{notation}}
\def\br{\begin{remark}}
\def\er{\end{remark}}
\def\bp{\begin{prop}}
\def\ep{\end{prop}}
\def\bpr{\begin{proof}}
\def\epr{\end{proof}}
\def\bt{\begin{thm}}
\def\et{\end{thm}}
\def\be{\begin{equation}}
\def\ee{\end{equation}}
\def\bl{\begin{lem}}
\def\el{\end{lem}}
\def\bc{\begin{cor}}
\def\ec{\end{cor}}
\def\bd{\begin{defn}}
\def\ed{\end{defn}}

\numberwithin{equation}{subsection}

\author{Mabrouk Ben Nasr}
\thanks{}
\address{Department of Mathematics, Faculty of sciences, Sfax, Tunisia}
\email{mabrouk$\_$bennasr@yahoo.fr}  

\author{Hassen Kthiri}
\thanks{}
\address{Department of Mathematics, Faculty of sciences, Sfax, Tunisia}
\email{hassenkthiri@gmail.com}

\author{Jean-Louis Verger-Gaugry}
\thanks{}
\address{
LAMA, CNRS UMR 5127,
Univ. Grenoble Alpes, Univ. Savoie Mont~Blanc,
F - \!73000 Chamb\'ery, \!France}
\email{Jean-Louis.Verger-Gaugry@univ-smb.fr}

\begin{document}

\dedicatory{\today}


\def\dyg{{\rm dyg}}


\title{On algebraic integers which are 
2-Salem 
elements in positive characteristic}

\markboth{Mabrouk Ben Nasr, 
Hassen Kthiri {\rm and} 
Jean-Louis Verger-Gaugry}{2-Salem 
elements in positive characteristic}

\begin{abstract}
Bateman and Duquette have initiated the study of Salem elements in positive characteristic. This work
extends their results to 2-Salem elements 
whose minimal polynomials are of the type 
$Y^n+\lambda_{n-1}Y^{n-1}+\ldots+\lambda_1Y+\lambda_0 
\in \mathbb{F}_q[X][Y ]$ where 
$n \geq2, \lambda_0\neq 0$ and
$\deg \lambda_{n-1} < \deg \lambda_{n-2} 
= \displaystyle\max_{i\neq n-2}\deg(\lambda_i)$.
This work provides an analogue of their results for 
2-Salem elements whose minimal polynomials
meet certain requirements.
\end{abstract}

\maketitle

\vspace{0.7cm}


\vspace{0.5cm}


\tableofcontents  

\newpage
\section{Introduction}
\label{S1}

A \textit{Salem number} is a real algebraic integer 
$\theta > 1$ of even degree at least 4, having 
$\theta^{-1}$ as a conjugate over $\mathbb{Q}$, 
having all its
conjugates $\theta_i$ excluding $\theta$ 
and  $\theta^{-1}$, of modulus exactly $1$ \cite{138}. 
The monic minimal polynomial, over 
$\mathbb{Q}$, $\Lambda(z)$ of  a 
Salem number  $\theta$ is
reciprocal: it satisfies the equation $z^{\deg \Lambda(z)}\Lambda(\frac{1}{z}) = \Lambda(z)$. To put is simply, this means that its coefficients form
a palindromic sequence: they read the same backwards as forwards. Therefore $\theta+\theta^{-1}$ is a real algebraic
integer $\theta> 2$ such that its conjugates $\neq\theta+\theta^{-1}$ lie in the real interval $[-2,2]$.
The Mahler measure
$M(\theta) :=\displaystyle\prod_{i=1}^{\deg \theta}\max\{1,|\theta_i|\}$
of $\theta$ satisfies $M(\theta) = \theta$. 
A Salem number is the Mahler measure of itself.

The {\it set of Salem numbers} is 
traditionally denoted by $T$ \cite{2}. 
The smallest known element  of $T$ is Lehmer's
number $\beta_0 = 1.1762\ldots$ of degree $10$ , 
as {\it dominant} root (``i.e. if $\beta$ is another root, 
then $|\beta|<\beta_0$") of Lehmer's polynomial:
\begin{equation}\label{0900}
P(X) = X^{10}+X^9-X^7-X^6-X^5-X^4-X^3+X+1.
\end{equation}
The
Surveys \cite{138} \cite{139} take stock of 
various problems on Salem numbers 
and more generally Mahler measures in all their forms.

Kerada \cite{6} defined and studied, 
as a generalization of a \textit{Salem number}, 
j-\textit{Salems}, 
$j \geq2$
(also called {\it $j$-Salem numbers} in the literature,
e.g. in \cite{umemoto}). 
In particular, a {\it 2-Salem} 
is a pair $(\beta_1, \beta_2)$ 
of conjugate algebraic integers of
modulus $> 1$ whose remaining conjugates have modulus 
at most $1$, with at least one having modulus exactly
$1$. 
The {\it set of 2-Salems}  is denoted by $T_2$. 
It is partitioned as $T_2 = T'_2\cup T''_2$
where $T'_2$ is the set of 2-Salems 
with $\beta_1,\beta_2\in \mathbb{R}$ 
and $T''_2$
the set of 2-Salems for 
which $\beta_1$ and $\beta_2$ are 
complex  non-real (and so
complex conjugates of one another, 
$\beta_1 = \bar{\beta_2}$).

In 1962 Bateman and Duquette \cite{1}  
introduced and characterized the Salem and Pisot (PV)
elements in the field of Laurent series.
We first start recalling their theorem prior
to stating an analogue theorem for 2-Salem elements,
extending Kerada's study.
\begin{theorem}[Bateman - Duquette] 
\label{01}
 An element $\omega$ in $\mathbb{F}_q((X^{-1}))$ 
 is a Salem (resp. Pisot) element if and only if
its minimal polynomial can be written 
$$\Lambda(Y )= Y^s+ \lambda_{s-1}Y^{s-1}+\ldots+ 
\lambda_0,\qquad
\lambda_i \in \mathbb{F}_q[X] \quad {\rm for} 
~i = 0,\ldots, s-1,$$ with
$|\lambda_{s-1}| = |\omega| > 1$ and
$|\lambda_{s-1}| =
\displaystyle\max_{0\leq i\leq s-2}|\lambda_i|$ 
$($resp. $|\lambda_{s-1}| 
>\displaystyle\max_{0\leq i\leq s-2}|\lambda_i|)$.
\end{theorem}

In this work,  instead of the classical setting of the real 
numbers, the analogues of Kerada's
2-\textit{Salems} over
the ring of formal Laurent series over finite fields are 
investigated. 
In the context of the original
study of \textit{Salem elements} in positive 
characteristic by
Bateman and Duquette \cite{1}, 
2-\textit{Salem elements} in positive 
characteristic
will also be called \textit{2-Salem series}. 
The objectives of the present note consist
in extending some of the results of 
Bateman and Duquette to 2-Salem series over $\mathbb{F}_q[X],~q\neq 2^r$, 
and to study
the analogues of the above-mentioned properties 
of 2-Salem series. 
More precisely, let  $\mathbb{F}_q$ denote 
the finite field
having q elements, $q \geq3$, and let p be the 
characteristic of  $\mathbb{F}_q$; $q$ is a power of $p$. 
Let $X$ be an indeterminate over
$\mathbb{F}_q$ and denote $k :=\mathbb{F}_q(X)$. 
Let $\infty$ be the unique place of $k$ which 
is a pole of $X$, and denote
 $k_\infty :=  \mathbb{F}_q(( \frac{1}{X}))$.
 Let $C_\infty$ be a completion of an algebraic closure 
 of $k_{\infty}$. Then $C_{\infty}$ is algebraically closed 
 and complete, and we
denote by $\upsilon_{\infty}$ the valuation on $C_{\infty}$ 
normalized by $\upsilon_{\infty}(X) =-1$. 
We fix an embedding of an algebraic closure
of $k$ in $C_{\infty}$ so that all the finite extensions of 
$k$ mentioned in this work will be contained 
in $C_{\infty}$. 
An explicit
description of $\upsilon_{\infty}$ is done in section 2. 
For simplicity's sake the algebraic closure of $k_{\infty}$ 
will be often denoted
by  $\mathbb{F}_q((X^{-1}))$.

2-Salem series over $\mathbb{F}_q[X]$ may belong to $k_{\infty}$ or to finite extensions of $k_{\infty}$. By analogy with Kerada's
notations we denote by $T^{\ast}_2$ 
the set of 2-Salem series. 
It can be partitioned as 
$T^{\ast}_2=T'^{\ast}_2\cup T''^{\ast}_2$
where $T'^{\ast}_2$ is by definition 
those 2-Salem series 
$(\omega_1,\omega_2)$ over $\mathbb{F}_q[X]$ 
which (both) belong to $\mathbb{F}_q((X^{-1}))$, 
and $T''^{\ast}_2$, by definition, those 2-Salem series, 
not in $\mathbb{F}_q((X^{-1}))$, 
such that $(\omega_1^n,\omega_2^n)\in T'^{\ast}_2$
for some integer $n\geq2$.

\begin{theorem}\label{1128} 
Suppose $q\neq 2^r$ for any integer $r\geq1$,
and $n \geq 3$. 
Let $\Lambda$ be the polynomial defined by
\begin{equation}\label{1h}
\Lambda(Y)= Y^n+\lambda_{n-1}Y^{n-1}+\lambda_{n-2}Y^{n-2}+\ldots+ \lambda_1Y+ \lambda_0\in \mathbb{F}_q[X][Y],
\end{equation}
assumed irreducible and
such that
$\lambda_0 \neq 0$.
Let us assume
$$\deg\lambda_{n-1}
<\deg\lambda_{n-2}
=\max_{i\neq n-2}\deg(\lambda_i).$$ 
Denote by $\omega_1$ and $\omega_2$ 
the dominant roots of $\Lambda$. 
Then
\begin{itemize}
  \item[(i)] for $n \geq 4$:
  if $ \deg\lambda_{n-2}>2\deg\lambda_{n-1}$, then  
  $(\omega_1,\omega_2)\in T'^{\ast}_2$
if and only if  $\deg\lambda_{n-2}$ is even, 
the dominant coefficient  $\alpha_{2s}$ of 
$\lambda_{n-2} 
= \alpha_{2s}X^{2s} +\ldots + \alpha_{0}$ 
is equal to $-a^2$ for some $a \in \mathbb{F}_{q}^{*}$,
and $\deg\lambda_{n-3}<\deg\lambda_{n-2}$,
\item[(ii)] for $n=3$:
if $ \deg\lambda_{1}>2\deg\lambda_{2}$, then  
  $(\omega_1,\omega_2)\in T'^{\ast}_2$
if and only if  $\deg\lambda_{1}$ is even, 
the dominant coefficient $\alpha_{2s}$ of 
$\lambda_{1} 
= \alpha_{2s}X^{2s} +\ldots + \alpha_{0}$ 
is equal to $-a^2$ for some $a \in \mathbb{F}_{q}^{*}$,
  \item[(iii)] for $n \geq 3$:
  if $\deg\lambda_{n-2}<2\deg\lambda_{n-1}$, then $(\omega_1,\omega_2)\in T'^{\ast}_2$.
\end{itemize}
 \end{theorem}

The paper is organized as follows. In section 2 the fields 
of formal power series and 
the valuations used in this 
study are recalled. The main  
Theorem \ref{01} of 
Bateman and Duquette, characterizing Salem elements,
is stated above with these notations. 
Section 3 is devoted 
to the arithmetical and topological properties 
of 2-Salem series in 
$\overline{\mathbb{F}_q((X^{-1}))}$. In
section 4 Weiss's method of the upper 
Newton polygon is explicited to characterize 
2-Salem series in 
$\overline{\mathbb{F}_q((X^{-1}))}$.
In section 5 attention is focused on 
those 2-Salem series which lie in the 
field $\mathbb{F}_q((X^{-1}))$, by 
establishing criteria
discriminating whether they belong 
to $\mathbb{F}_q((X^{-1}))$ or to 
$\overline{\mathbb{F}_q((X^{-1}))}\setminus \mathbb{F}_q((X^{-1}))$. 
The proof of Theorem \ref{1128} is
given in section 6. 
In Theorem \ref{1128}  the polynomial 
given by \eqref{1h} is assumed irreducible. 
More generally, the
question of irreducibility of a polynomial 
$\Lambda$ of the 
general form \eqref{1h} is discussed in section~7
under the hypothesis that $\Lambda$ has no root in
$\mathbb{F}_q$.

\section{ Salem series in  $\mathbb{F}_q((X^{-1}))$}
For $p$ a prime and $q$ a power of $p$, let $\mathbb{F}_q((X^{-1}))$ be the set of Laurent series over $\mathbb{F}_q$  which is defined as follows
$$\mathbb{F}_q((X^{-1}))=\{\omega=\sum_{i\geq n_0}\omega_i X^{-i}:
n_0 \in \mathbb{Z} \;\; {\rm and}\;\;  \omega_i \in \mathbb{F}_q\}.$$
We know that every algebraic element over $\mathbb{F}_q[X]$ can be written explicitly as a formal series because $\mathbb{F}_q[X]\subseteq \mathbb{F}_q((X^{-1}))$.
However, as $\mathbb{F}_q((X^{-1}))$ is not algebraically closed, such an element is not
necessarily expressed as a power series. We refer to Kedlaya \cite{5} for a full characterization of the
algebraic closure of $\mathbb{F}_q[X]$.
 We denote by  $\overline{\mathbb{F}_{q}((X^{-1}))}$ an algebraic closure of $\mathbb{F}_{q}((X^{-1}))$.
Indifferently we will speak of 2-Salem elements or 
2-Salem series in the present context.

Let $\omega$ be an element of $\mathbb{F}_q((X^{-1}))$, its polynomial part is denoted
by $[\omega]\in\mathbb{F}_q[X]$ and  $\{\omega\}$ its fractional part. We can remark that $\omega = [\omega]+\{\omega\}$. If
$\omega\neq 0$, then the {\it polynomial degree} 
$\deg \omega$ of $\omega$ is 
$\gamma(\omega) = \sup\{-i : \,\,\omega_i\neq 0\}$, 
the
degree of the highest-degree nonzero monomial in $\omega$, 
with the convention $\gamma(0) =-\infty$.  
The generic form of
$\omega$, with $n_0\in \mathbb{Z}$ and 
$ \omega_i\in \mathbb{F}_q$,
$ n_0 = -\gamma(\omega)$, is
$$\omega=\sum_{i\geq n_0}\omega_i X^{-i}.$$
Note
that if $[\omega]\neq 0$ then $\gamma(\omega)$ 
is the degree of the polynomial $[\omega]$. 
Thus, we define
 the absolute value
$$|\omega|=\left\{%
\begin{array}{ll}
    q^{\gamma(\omega)} & \hbox{for $ \omega \neq  0$;} \\
    0  & \hbox{for $ \omega=0$.} \\
\end{array}%
\right.    $$
 Since $|.|$ is not archimedean, $|.|$ fulfills the strict
 triangle inequality
\begin{eqnarray*}
   |\omega+\nu| &\leq& \max \ (|\omega|,|\nu|) \ \ \ \ \ \ and \\
   |\omega+\nu| &=& \max \ (|\omega|,|\nu|) \ \ \ \ \ \ if \ |\omega|\neq|\nu|.
\end{eqnarray*}
\begin{definition}
A {\it Salem element} $\omega$ in $\mathbb{F}_q((X^{-1}))$ 
is an algebraic integer over $\mathbb{F}_q[X]$ such
that $|\omega| > 1$, whose remaining conjugates 
in $\overline{\mathbb{F}_{q}((X^{-1}))}$ have 
an absolute value
no greater than $1$, 
and at least one has absolute value exactly $1$.
A {\it Pisot element} $\omega$ in 
$\mathbb{F}_q((X^{-1}))$ is an algebraic integer 
over $\mathbb{F}_q[X]$ such
that $|\omega| > 1$, whose remaining conjugates 
in $\overline{\mathbb{F}_{q}((X^{-1}))}$ have 
an absolute value
strictly less than $1$.
The set of Salem elements, resp. Pisot elements, is denoted $T^{\ast}$, resp. $S^\ast$.
\end{definition}

In the following we will 
focus on 2-Salem series in $k_{\infty}$: 
a {\it 2-Salem element} is
a pair of 
series 
$(\omega_1,\omega_2)$  in
$\mathbb{F}_q((X^{-1})) \times \mathbb{F}_q((X^{-1}))$,
which has an absolute value greater than $1$, 
in the sense that it is such that $\omega_1$
is an algebraic integer over $\mathbb{F}_q[X]$, 
with the property
that all of its conjugates 
$\omega_i$, $i \neq 1,2$,
lie on or within the 
unit circle, and at least one
conjugate lies on the unit circle.  
This implies that all 2-Salem elements are necessarily 
separable over  $\mathbb{F}_q(X)$.
Note that the pair
$(\omega_1,\omega_2)$ is not ordered.

Let us remark that it is easy to construct a 2-Salem element over $\mathbb{F}_q$ with $q=2$ and 
then to show that
2-Salem elements do exist without the 
assumption $q\neq2^r, r\geq1$, taken in 
Theorem \ref{1128}. 
The exclusion case $q\neq 2^r$ of 
Theorem \ref{1128} will arise in a general 
setting from Lemma \ref{1211} and its consequences.

\section{Multiplicative properties of 2-Salem series}

\begin{proposition}
\label{paire_n}
Let $(\omega_{1},\omega_{2})\in T'^{\ast}_2$, then
$( \omega_{1}^n, \omega_{2}^n)\in T'^{\ast}_2$, 
for all $ n \geq 1$.
\end{proposition}
\begin{proof}
Let $M \in \mathbb{F}_q[X][Y]$ the minimal 
polynomial of the algebraic integer 
$\omega = \omega_1$ of degree $d$ and 
$\omega_{2},\ldots, \omega_{d}$ the conjugates of $\omega$. We consider that the conjugate
$\omega_{2}$
of $\omega_1$ is the
only conjugate which lies outside the unit disk.
Evidently,
since $\omega_1$ is an algebraic integer 
over $\mathbb{F}_q[X]$, 
$\omega_{1}^n$, for $n \geq 1$,  
is also an algebraic integer
over $\mathbb{F}_q[X]$.

Let $n \geq 1$ and
$\Lambda\in \mathbb{F}_q[X][Y]$ be the minimal polynomial of
$ \omega_{1}^n$. 
We consider the embedding $\sigma_i$ of
$\mathbb{F}_q(X)( \omega_{1})$ into $\overline{\mathbb{F}_q((X^{-1}))}$, which fixes
$\mathbb{F}_q(X)$ and maps $ \omega_{1}$ to $ \omega_{i}$.
Obviously, for
$i=1, 2, \ldots, d$, $ \omega_{i}^n$ is a root of the equation $\Lambda(Y)=0$,
and 
$ \omega_{1}^n, \omega_{2}^n,\ldots, \omega_{d}^n$
are all the roots of $\Lambda$, since
$$\Lambda( \omega_{i}^n) =
\Lambda((\sigma_i( \omega_{1}))^n))
=\Lambda(\sigma_i( \omega_{1}^n))
=\sigma_i(\Lambda( \omega_{1}^n))
=\sigma_i(0) =0.$$
We deduce $\deg(\Lambda)\leq \deg(M)$
since
$$[\mathbb{F}_q(X)( \omega_{1}^n):\mathbb{F}_q(X)]\leq
[\mathbb{F}_q(X)( \omega_{1}): \mathbb{F}_q(X)].$$
If $3\leq i \leq d$,  then
$| \omega_{i}^n|=| \omega_{i}|^n\leq1$ and 
there exists at least one $j$, $ 3\leq j\leq n$, 
such that
$| \omega_{j}^n|=| \omega_{j}|^n=1$.
 Therefore $( \omega_{1}^n, \omega_{2}^n) \in T'^{\ast}_2$, for all $n \geq 1$.
\end{proof}
Note that the converse is false in general. For instance, take $q = 3, d =4$  and $n=2$. Then, the polynomial
$$Y^4-2X^2Y^2+2X^2$$
over $\mathbb{F}_3$  is irreducible and  its two roots of absolute value $> 1$ defined by
\begin{equation*}
    (\omega_1,\omega_2)=((\sqrt{2}(X-\displaystyle\frac{1}{X^3}+\ldots),-(\sqrt{2}(X-\displaystyle\frac{1}{X^3}+\ldots)),
\end{equation*}
not lie in $\mathbb{F}_3((X^{-1}))$. The other conjugates defined by
\begin{equation*}
    (\omega_3,\omega_4)=(1-\displaystyle\frac{1}{X^2}+\ldots,-(1+\displaystyle\frac{1}{X^2}+\ldots)).
\end{equation*}
We can see that $(\omega_1^2,\omega_2^2)$ lie in $\mathbb{F}_3((X^{-1}))$.

For a 2-Salem series $\theta$, let us define
its trace by ${\rm Tr}(\theta) :=\displaystyle\sum_{i=1}^{\deg \theta}\theta_i$.
The 2-Salem series have the following basic property, 
as it can easily be seen by considering its trace.
Recall that in the real case the trace of a Salem number
is an integer $(\in \mathbb{Z})$
which is not bounded and can take
arbitrary negative values \cite{16}.
\begin{proposition}\label{1321}
Let $( \omega_{1}, \omega_{2}) \in T'^{\ast}_2$, 
then the sequence 
$(\displaystyle\{ \omega_{1}^n+ \omega_{2}^n\})_{n \geq 1}$ is bounded.
 \end{proposition}
\begin{proof}
Let  $( \omega_{1}, \omega_{2})$ 
be a  2-Salem  series and 
$ \omega_{3}, \ldots, \omega_{d}$ the other
conjugates of  $\omega_{1}$ and  $\omega_{2}$. 
From Proposition \ref{paire_n}, 
for all $n\geq 1$, $\omega_{1}^n$ 
and $\omega_{2}^n $ are the
roots of the same irreducible polynomial, 
say 
$\Lambda_n$ in
$\mathbb{F}_q[X]$,
of degree $d$. 
We have
$${\rm Tr}(\Lambda_n)=\displaystyle\sum_{i=1}^d \omega_{i}^n 
\in \mathbb{F}_q[X] .$$
Thus $\{{\rm Tr}(\Lambda_n)\}=0$, which can
be rewritten 
$$0=
\{{\rm Tr}(\Lambda_n) =\displaystyle\sum_{i=1}^d \omega_{i}^n\}
=\{ \omega_{1}^n+ \omega_{2}^n+
\displaystyle\sum_{i=3}^d \omega_{i}^n\}.$$
But
$| \omega_{i}|\leq1$,
for $3\leq i\leq d$,
and there exists at least one $j$,   
$3\leq j\leq n$ such that
$| \omega_{j}^n|=| \omega_{j}|^n=1$.
Therefore, taking the absolute values, 
we deduce
$|\{\omega_1^n+\omega_2^n\} |
= |\{\omega_3^n\}+\{\omega_4^n\}+\ldots+\{\omega_d^n\}|$
and
$$\lim_{n\mapsto+\infty} 
\Bigl|
\{\sum_{i=3}^d \omega_i^n\}
\Bigr|
\leq 
\max_{i=3,\ldots,d}
\bigl\{
\bigl|\{\omega_i^n\}
\bigr|
\bigr\}
 \leq 
 C
 \in \mathbb{F}_q ,
 $$
and then $\{\omega_1^n+\omega_2^n \}$ is bounded.
\end{proof}

\begin{remark}
If the 
2-Salem series 
$( \omega_{1}, \omega_{2})\in T'^{\ast}_2$ 
of Proposition
\ref{1321} admits only one root $\omega_3$ 
having absolute value equal to $1$ 
and for which the other conjugates
have an absolute value strictly less than $1$,
then $\displaystyle\lim_{n\rightarrow+\infty}
\{\omega_1^n+\omega_2^n\}=0$.
\end{remark}
\begin{proof}
It is a consequence of the
definition of the upper Newton polygon
of the polynomial $\Lambda_n$,
recalled in Proposition \ref{03} below.
From Proposition \ref{03} 
we can see that $\omega_3\in \mathbb{F}_q((X^{-1}))$. 
Thus
\begin{equation}\label{987}
    \displaystyle\lim_{n\mapsto+\infty}\{\omega_3^n\}=0.
\end{equation}
From the proof of Proposition \ref{1321}, 
we have
$\omega_1^n+\omega_2^n=
{\rm Tr}(\omega_1^n)-\omega_3^n-\omega_4^n-\ldots-\omega_d^n$,
$n \geq 1$,
what implies
\begin{eqnarray*}
  |\{\omega_1^n+\omega_2^n\} |&=& |\{\omega_3^n\}+\{\omega_4^n\}+\ldots+\{\omega_d^n\}| \\
   &\leq& |\{\omega_3^n\}+ \omega_4^n+\ldots+\omega_d^n|\\
  &\leq&\displaystyle\max_{i=4,\ldots,d}\{|\{\omega_3^n\}|,|\omega_i^n|\}.
\end{eqnarray*}
Since $|\omega_i| < 1$ for $i = 4, \ldots, d$ 
and by \eqref{987}, the assertion of the Remark follows.
\end{proof}
\begin{proposition}
Let $( \omega_{1}, \omega_{2})\in T'^{\ast}_2$ 
be a 2-Salem series.
Assume that 
$\Lambda\in \mathbb{F}_q[X][Y]$ is
its minimal polynomial,
that the degree of $\Lambda$ is equal to 
$4$ and
$\omega_{1}, \omega_{2}, \omega_{3},
\omega_{4}$ are its four roots, 
the root $\omega_3$ satisfying
$\deg \omega_{3}= 0$.
If $\Lambda(0)\in  \mathbb{F}_q^\ast$, then
$\omega_{1} \omega_{2}\omega_{3}\in T^{\ast}$.
\end{proposition}
\begin{proof}
Let $(\omega_{1},\omega_{2})\in T'^{\ast}_2$ and
$$\Lambda(Y)= Y^4+\lambda_3 Y^3+\lambda_2 Y^2+\lambda_1 Y +\lambda_0 ,
\qquad \lambda_0 \in \mathbb{F}_{q}^{*},$$
the minimal polynomial  of $\omega_{1}$ and $\omega_{2}$. 
We have 
$\lambda_0 =
\omega_1 \omega_2 \omega_3 \omega_4$.
Consider the reciprocal polynomial of $\Lambda$
$$Q(Y):=Y^{4}\Lambda( \displaystyle\frac{1}{Y}).$$
Clearly $Q$ is an irreducible polynomial
over $\mathbb{F}_q[X]$,
and admits the four roots
$$\frac{1}{\omega_1},~\frac{1}{\omega_2},~\displaystyle\frac{1}{\omega_3}
= \lambda_{0}^{-1}
\omega_{1} \omega_{2}\omega_{4}~
\quad {\rm and}\quad
\frac{1}{\omega_4}= \lambda_{0}^{-1}
\omega_{1} \omega_{2} \omega_{3}.$$
We have
$$
   |\displaystyle\frac{1}{\omega_3}|=|\omega_1\omega_2\omega_4|
=1,\qquad
       |\displaystyle\frac{1}{\omega_4}|=|\omega_1\omega_2\omega_3|>1
$$
and
   $|\displaystyle\frac{1}{\omega_i}|<1$,
for $i=1,2$.
Therefore $\omega_1\omega_{2}\omega_{3}
=\frac{\lambda_0}{\omega_4}$ is a
Salem series.
\end{proof}

\section{A first characterization of 2-Salem series}

The theory of the Newton polygon of a bivariate 
polynomial is used in the present study. 
The following Proposition of Weiss in \cite{11} is 
the main tool for our purposes.
Let us recall it.
Let
 \begin{equation}
    \Lambda(X,Y)= \lambda_nY^n+\lambda_{n-1}Y^{n-1}+\ldots+ \lambda_1Y+\lambda_0 \quad 
    \in \mathbb{F}_q[X,Y] = \mathbb{F}_q[X][Y]
 \end{equation}
be a nonzero polynomial.
To each monomial $\lambda_iY^i\neq0$, 
we assign the point
$(i,\deg(\lambda_i))\in \mathbb{Z}^2$. 
For $\lambda_i=0$, we ignore the corresponding 
point $(i,-\infty)$.
If we consider the upper convex hull of the set of points
$$\{(0, deg(\lambda_0)),\ldots , (n, deg(\lambda_n))\},$$
we obtain the upper Newton polygon 
of $\Lambda(X,Y)$ with respect to $Y$.
The polygon is a sequence of line segments 
$E_1, E_2,\ldots E_t$, with monotonous
decreasing slopes.

The slope of a segment of the Newton polygon of 
$\Lambda(X,Y)$ joins, for instance, the point
$(r,\deg(A_r))$ to $(r+s, \deg(A_{r+s}))$ 
for some  $0\leq r<r+s \leq m$.
The corresponding slope is
$$k = \displaystyle\frac{\deg(A_{r+s})-\deg(A_r)}{s}.$$
Denote by $K_{\Lambda}$  the set of the slopes.
For any slope $k \in K_{\Lambda}$, denote by 
$s$ the length of the facet of slope $k$.
 
\begin{proposition}[Weiss]\label{03}
Let
$$
   \Lambda(X,Y)=Y^n+\lambda_{n-1}Y^{n-1}+\ldots+ 
   \lambda_1Y+\lambda_0
   \in \mathbb{F}_q[X,Y]
$$
and $K_\Lambda$ the set of the slopes of its upper 
Newton polygon. Then, for every $k\in K_\Lambda$,
\begin{itemize}
  \item[i)] $\Lambda(X,Y)$, as a polynomial in $Y$, has $s$ roots $\alpha_1,\ldots, \alpha_{s}$ with the same degree $-k$ and
$$|\alpha_1| = \ldots =|\alpha_{s}| = q^{-k} ,$$
  \item[ii)] the polynomial
  $$
    \Lambda_k(X,Y) =  \displaystyle \prod_{i=1}^{{s}}(Y-\alpha_i)
    \quad\in \mathbb{F}_q((X^{-1}))[Y]
  $$
  divides $\Lambda(X,Y)$, with
$$
    \Lambda(X,Y)= \displaystyle \prod_{k\in K_\Lambda} \Lambda_k(X,Y).
$$
\end{itemize}
\end{proposition}
Corollary \ref{04} is an application of Proposition
\ref{03} obtained by Ben Nasr and Kthiri in \cite{7} 
in the context of 2-Pisot elements. 
In the case of 2-Salem elements,
Proposition \ref{03} has several direct consequences:
the following Corollary \ref{05} and
Theorem \ref{06}.

\begin{corollary}\label{04}
Let
\begin{align}
    \Lambda(X,Y)= \lambda_nY^n+\lambda_{n-1}Y^{n-1}+\ldots+ \lambda_1Y+\lambda_0\in \mathbb{F}_q[X][Y].
 \end{align}
 and $\omega$ a root of $\Lambda$. If $|\lambda_n|=\displaystyle\max_{0\leq k\leq n}|\lambda_k|$, then
$|\omega|\leq1$.
\end{corollary}
\begin{corollary}\label{05}
Let $n \geq 3$.
Let \begin{equation}
    \Lambda(X,Y)= Y^n+\lambda_{n-1}Y^{n-1}+\ldots+ \lambda_1Y+\lambda_0\in \mathbb{F}_q[X][Y]
 \end{equation}
with $\lambda_0 \neq 0$. 
Let us assume
$$\deg \lambda_{n-1}<
\displaystyle
\max_{0\leq k<n-2}\deg\lambda_k =\deg\lambda_{n-2} < 2
\deg \lambda_{n-1}.$$
Then, $\Lambda$ has only two roots 
$\omega_1, \omega_2 \in \mathbb{F}_q((X^{-1}))$  
satisfying $|\omega_1|> 1$ and $|\omega_2|>1$
and at least one conjugate
which lies on the unit circle.
\end{corollary}
\begin{proof}
First let us notice that the stated condition 
implies that $\deg \lambda_{n-1}>0$. 
Moreover the upper Newton polygon of $\Lambda$ 
contains the line
with a slope $k_1$ joining $(n-1,\deg\lambda_{n-1})$ 
and $(n,0)$, 
the line with a slope $k_2$ joining 
$(n-2,\deg\lambda_{n-2})$ and  
$(n-1,\deg\lambda_{n-1})$ and
the line with a slope $k_3 = 0$ 
joining 
$(n-2,\deg\lambda_{n-2})$ and  
$(n-k,\deg\lambda_{n-k}= \deg\lambda_{n-2})$
for some $0\leq k < n-2$.
We have: $\deg\lambda_{n-2} - \deg \lambda_{n-1}
< \deg \lambda_{n-1}$.
By Proposition \ref{03} $(i)$,
$\Lambda$ has exactly  two 
dominant roots $\omega_1$, $\omega_2 $
$$\left\{
  \begin{array}{ll}
    |\omega_1|= & \hbox{$q^{\deg\lambda_{n-1}}
    =q^{-k_1}>1$}, \\
    |\omega_2|= & \hbox{$q^{\deg\lambda_{n-2}-\deg\lambda_{n-1}}=q^{-k_2}>1$.}
  \end{array}
\right.$$
There exists 
$0 \leq k < n-2$ such that
$\deg\lambda_k = \deg\lambda_{n-2}$;
hence $\Lambda$ has one root, say $\omega_3 $,
such that
\begin{equation*}
     |\omega_3|= 
     q^{\frac{-\deg\lambda_{n-2}+\deg\lambda_k}{n-2-k}}
     = q^{-k_3} = 1.
\end{equation*}

By Proposition \ref{03} $(ii)$, 
$\Lambda$ admits
the two factors $\Lambda_{k_1}(X,Y)
= (Y-\omega_1)\in\mathbb{F}_q((X^{-1}))[Y]$
and 
$\Lambda_{k_2}(X,Y)
= (Y-\omega_2)\in\mathbb{F}_q((X^{-1}))[Y]$. 
Hence
  $\omega_1$ and $\omega_2\in\mathbb{F}_q((X^{-1}))$.
\end{proof}

\begin{theorem}\label{06}
Let $\Lambda$ be the polynomial of degree 
$n\geq3$ defined by
$$ \Lambda(Y)=
Y^n+\lambda_{n-1}Y^{n-1}+\lambda_{n-2}Y^{n-2}+\ldots+
\lambda_1Y+\lambda_0 \in \mathbb{F}_q[X][Y ]$$
with $\lambda_0 \neq 0$.
Then,
$\Lambda$ has exactly $2$ roots in  
$\overline{\mathbb{F}_{q}((X^{-1}))}$ which have 
an absolute value
strictly greater than $1$
and the remaining roots in  
$\overline{\mathbb{F}_{q}((X^{-1}))}$   
which have  an absolute 
value less or equal to $1$, with 
at least one conjugate
lying on the unit circle, if and only if 
the following conditions are satisfied:
$|\lambda_{n-1}|<|\lambda_{n-2}| 
= \displaystyle\max_{0 \leq i< n-2}|\lambda_i|$.
\end{theorem}
\begin{proof}
Let $\omega_1,\omega_2,\ldots, \omega_n$ 
be the roots of $\Lambda$.
The conditions are necessary.
Suppose  
$|\omega_1|\geq|\omega_2|>1
\geq|\omega_{3}|\geq\ldots\geq|\omega_n|$ 
and that there exists at least one $j$, 
$3\leq j \leq n$, such that
 $|\omega_{j}|=1$. 
We have $ |\lambda_{n-2}| > |\lambda_{n-1}|$.
 For  
 $k\in \{1,\ldots,n\}$, $k\neq2$, 
$$|\lambda_{n-k}|=
\Bigl|\displaystyle\sum_{1\leq i_1<i_2<\ldots <i_k\leq n}
\omega_{i_1}\omega_{i_2} \ldots\omega_{i_{k}}
\Bigr|
~\leq~
|\omega_{1}\omega_{2} \ldots\omega_{k}|
~\leq~
|\omega_1\omega_2|
=|\lambda_{n-2}|$$
and
$$|\lambda_{n-j}|
=
\Bigl|\sum_{1\leq i_1<i_2<\ldots <i_j\leq n}
\omega_{i_1}\omega_{i_2} \ldots\omega_{i_{j}}
\Bigr|
=
|\omega_{1}\omega_{2} \ldots\omega_{j}|
=
|\omega_1\omega_2|
=|\lambda_{n-2}|.$$
Then $$|\lambda_{n-2}|=\displaystyle\max_{i\neq n-2}|\lambda_i|.$$
The conditions are sufficient.
The converse easily follows from Proposition \ref{03}.
\end{proof}
\begin{example}\label{24}
\end{example}
Let
$$\Lambda(Y)=Y^3+(X+1)Y^2+(X^4+X^3)Y +X^4+X^3+X^2+X+1\in \mathbb{F}_2[X][Y].$$
By Theorem \ref{06},
$\Lambda(Y)$ has two roots  
$\omega_1$ and $\omega_2$ having 
absolute value  strictly greater than $1$ 
and one root $\omega_3$
which has an absolute value exactly equal to $1$. 
Using the facts that
\begin{itemize}
  \item $[ \omega_1+\omega_2 +\omega_3]=X+1,$
  \item  $[\omega_1\omega_2+\omega_1\omega_3+\omega_2\omega_3 ]=X^4+X^3,$
  \item $[\omega_1\omega_2\omega_3]=X^4+X^3+X^2+X+1$,
\end{itemize}
then $\omega_1, ~\omega_2$ and $\omega_3$ are defined by:
$$\left\{
  \begin{array}{ll}
    \omega_1= & \hbox{$X^2 +1+ \displaystyle\frac{1}{Z_1}$ ~~such that $|Z_1|>1$,} \\
    \omega_2= & \hbox{$X^2+X+\displaystyle\frac{1}{Z_2}$ ~~such that $|Z_2|>1$,} \\
  \end{array}
\right.$$
and $\omega_3=1+\displaystyle\frac{1}{Z_3}$ 
such that $|Z_3|>1$. For $j=1$, resp. $j=2$,
the fact that
$\Lambda(\omega_j)=0$ implies
that $Z_1$, resp.  $Z_2$,  is a root of 
the polynomial $H_1$, resp. $H_2$, defined by
\begin{equation}\label{26}
 H_1= Z^3+(X^3+1)Z^2+(X^2+X)Z+1,
\end{equation}
resp.
\begin{equation}\label{27}
 H_2= (X^2+X+1)Z^3+(X^3+X^2)Z^2+(X^2+1)Z+1.
\end{equation}
Applying Proposition \ref{03} 
to the equations \eqref{26} and \eqref{27}, 
we obtain  $Z_1,Z_2\in\mathbb{F}_{2}((X^{-1}))$.
Therefore $\omega_1,\omega_2\in\mathbb{F}_{2}((X^{-1}))$.
Since $\Lambda$ is monic and irreducible over 
$\mathbb{F}_2[X]$, 
we deduce that $(\omega_1,\omega_2)$ is  
a 2-Salem series and $\Lambda$ is the minimal 
polynomial of $\omega_1$.

\section{Criteria of existence of roots and conjugates in $\mathbb{F}_q((X^{-1}))$}
Before giving the proof of our results, 
we establish some lemmas that will be needed.
\begin{lemma} \label{07}
Let $n\geq3$. Let $\Lambda$ be defined by
$$\Lambda(Y)=Y^{n}+\lambda_{n-1}Y^{n-1}+\lambda_{n-2}Y^{n-2}
+\ldots+\lambda_{1}Y+\lambda_{0}\in\mathbb{F}_{q}[X][Y],$$
with $\lambda_0 \neq 0$.
Suppose  
$\displaystyle\max_{i\neq n-2}\deg\lambda_i=\deg\lambda_{n-2}\geq2\deg\lambda_{n-1}$.
If  $\deg(\lambda_{n-2})$ is odd, 
then $\Lambda$ has no root
in  $\mathbb{F}_q((X^{-1}))$ with absolute value $ > 1$.
\end{lemma}
\begin{proof}
 By Theorem \ref{06}, $\Lambda$ has two roots $\omega_1$ and $\omega_2$ such that  $|\omega_1| > 1$ and $|\omega_2|>1$. The
remaining roots $\omega_3,\ldots,\omega_n$ have an absolute value less or equal to $1$ and at least one
conjugate 
$\omega_j$
lies on the unit circle for $3\leq j\leq n$.
As $\deg\lambda_{n-2}\geq2\deg\lambda_{n-1}$,
then  the upper Newton polygon of $\Lambda$ 
contains the line
connecting the points $(n-2,\deg\lambda_{n-2})$ 
and $(n,0)$. 
The slope of this line is 
$k=-\displaystyle\frac{\deg\lambda_{n-2}}{2}$.
By Proposition \ref{03} $(i)$,  
$\Lambda$ has  $n-(n-2)=2$ roots  $\omega_1$ and 
$\omega_2$
having the absolute value 
$ q^{-k}> 1$.  
Since they have the same  absolute value 
$ q^{-k}$, we would have
\begin{equation}
    \deg \omega_1= \deg \omega_2 =-k=\displaystyle\frac{\deg\lambda_{n-2}}{2}\notin \mathbb{Z}.
\end{equation}
Therefore $\omega_1,\omega_2\notin\mathbb{F}_q((X^{-1}))$.
\end{proof}

\begin{lemma}\label{1211}
Let $q\neq 2^r$ for any $r \geq 1$
and $n \geq 3$.
Let $\Lambda$ be the polynomial  defined by
\begin{equation*}
    \Lambda(Y)=Y^{n}+\lambda_{n-1}Y^{n-1}+\lambda_{n-2}Y^{n-2}+\ldots+\lambda_{1}Y+\lambda_{0}
    \in\mathbb{F}_{q}[X][Y]
\end{equation*}
with $\lambda_0 \neq 0$.
Suppose  
$\deg\lambda_{n-2}\geq\displaystyle\max_{i\neq n-2}\deg(\lambda_i)\quad 
{\rm and}\quad \deg\lambda_{n-2}>2\deg\lambda_{n-1}.$
Let $\omega_1$ be a root of $\Lambda$ 
such that $|\omega_1|>1$.
If $\displaystyle\deg\lambda_{n-3}
=
\deg\lambda_{n-2}$, 
then  $\omega_1\in \overline{\mathbb{F}_{q}((X^{-1}))}
\setminus\mathbb{F}_q((X^{-1}))$.
\end{lemma}

\begin{proof}
By Theorem \ref{06} the polynomial
$\Lambda$ has exactly two roots
$\omega_1$, $\omega_2$, such that
$|\omega_1|> 1, |\omega_2|> 1$, and at least one, 
$\omega_3$, such that
$|\omega_3|=1$.
According to Lemma \ref{07}, we conclude that 
$\deg\lambda_{n-2}$ is even. Set $\deg\lambda_{n-2}=2s
> 0$, then $\deg \omega_1= \deg \omega_2=s$. 
Let us assume $\omega_1 \in 
\mathbb{F}_q((X^{-1}))$.
Consider
\begin{equation}
\label{omega1omega2}
    \omega_1=\displaystyle\sum_{i=0}^s a_{i}X^{i}+\frac{1}{Z_1}\, , \qquad \,{\rm resp.}
    \qquad
    \omega_2=\displaystyle\sum_{i=0}^s b_{i}X^{i}+\frac{1}{Z_2}
\end{equation}
such that $a_s\neq0, b_s \neq 0$ 
and $|Z_1|>1, |Z_2|>1$. Let $\lambda_n=1$,
\begin{equation*}
    \lambda_i= \displaystyle\sum_{k_i=0}^{m_i}\alpha_{(k_i,i)} X^{k_i}
\end{equation*}
with $m_i\leq 2 s$ 
for $i=0, \ldots , n-4$,
$m_{n-3} = 2 s$,   and
\begin{equation*}
    \lambda_{n-2}=\displaystyle\sum_{j=0}^{2s}\alpha_{(j,n-2)}X^{j}
\end{equation*}
such that $\alpha_{(2s,n-2)}\neq0$.
We now prove that 
necessarily $|Z_1| \leq 1$, 
in contradiction
with $|Z_1| > 1$.

Indeed, the identity
$\Lambda (\omega_1)=0$ implies $0 =$
\begin{equation*}
\Bigl([\omega_1]+\frac{1}{Z_1}\Bigr)^n
+
\lambda_{n-1}
\Bigl([\omega_1]+\frac{1}{Z_1}
\Bigr)^{n-1}
+
\lambda_{n-2}
\Bigl([\omega_1]+\frac{1}{Z_1}
\Bigr)^{n-2} 
+
\ldots+
\lambda_{1}
\Bigl([\omega_1]+\frac{1}{Z_1}
\Bigr)
+
\lambda_{0}.
\end{equation*}
Multiplying it by $Z_1^n$, we  obtain
$$Z_1^n
\Bigl(\sum_{k=0}^n\lambda_{k}[\omega_1]^k
\Bigr)+ 
Z_1^{n-1}
\Bigl(\sum_{k=1}^n k\lambda_{k}[\omega_1]^{k-1}
\Bigr)
+Z_1^{n-2}
\Bigl(\sum_{k=2}^n\frac{k(k-1)}{2}\lambda_{k}
[\omega_1]^{k-2}
\Bigr)$$
$$+\ldots +Z_1^{n-j}
\Bigl(\sum_{k=j}^{n}
\frac{k (k-1)\ldots(k-j+1)}{j!}\lambda_{k}
[\omega_1]^{k-j}
\Bigr)
+\ldots+1=0.$$
Whence $Z_1$ is the root of 
the polynomial $H$  defined by
\begin{equation*}
H(Z)=A_n Z^n+A_{n-1}Z^{n-1}+ \ldots+1\in\mathbb{F}_{q}[X][Z]
\end{equation*}
where
\begin{align}\label{141}
A_i=\displaystyle\sum_{k=0}^i\binom{n-k}{i-k}\lambda_{n-k}[\omega_1]^{i-k},\quad 0\leq i\leq n.
\end{align}
Moreover
\begin{equation}\label{15}
   -\lambda_{n-1} = [\omega_1]+[\omega_2]
   + [\omega_3]
\end{equation}
and
\begin{align}\label{16}
   \lambda_{n-2}&=\omega_1\omega_2+\omega_1\omega_3+\ldots+\omega_{n-1}\omega_n\\
   &= [\omega_1][\omega_2]+Q
\end{align}
 with $Q \in \mathbb{F}_{q}[X]$ and $\deg Q \leq s-1$.
Notice that 
$\displaystyle\deg\lambda_{n-2} > 2 \deg\lambda_{n-1}$ implies
\begin{equation}\label{13}
    \deg \lambda_{n-1} = \deg([\omega_1]+[\omega_2])<s.
\end{equation}
then $a_s + b_s = 0$. 
Hence
$[\omega_1] - [\omega_2] = 
2 a_s X^s 
+ (a_{s-1} - b_{s-1} )X^{s-1} +
\ldots + ( a_0 - b_0 )$.
Since $q \neq 2^r$, for any $r \geq 1$, then 
$\deg([\omega_1] - [\omega_2]) = s$.
It follows from ~\eqref{15} and ~\eqref{16} 
that, for $0 \leq i \leq n$,
$0 \leq k \leq i$,
$$\left\{
  \begin{array}{ll}
    \deg(\lambda_{n-k}[\omega_1]^{i-k})=is & \hbox{$\quad {\rm for}~~ k=0,2$}, \\
   \deg(\lambda_{n-k}[\omega_1]^{i-k})< is & \hbox{$\quad {\rm for}~~ k\neq0,2$.}
  \end{array}
\right.
$$
Then
\begin{equation*}
    \deg A_i \leq is, \quad0\leq i\leq n.
\end{equation*}
In view of \eqref{141}, 
\eqref{15} and \eqref{16},
we can write
\begin{align*}
  A_n &=  [\omega_1]^{n}+\lambda_{n-1}[\omega_1]^{n-1}+
  \lambda_{n-2}[\omega_1]^{n-2}+
  \ldots+\lambda_0 \\
   &= 
-[\omega_3][\omega_1]^{n-1}+   
   [\omega_1]^{n-2}Q+\lambda_{n-3}[\omega_1]^{n-3}+\ldots+\lambda_0.
\end{align*}
Thus
\begin{equation*}
    \deg A_n=(n-1)s.
\end{equation*}
Again, by \eqref{141}, it is easy to show 
\begin{equation*}
  \deg A_i\leq (n-1)s, \quad~{\rm for}~ 0\leq i \leq n-1.
\end{equation*}
As a result, by applying 
Corollary \ref{04}, we obtain $|Z_1|\leq1$,
a contradiction.
\end{proof}
\section{ Proof of Theorem \ref{1128}}
For establishing the proof of Theorem \ref{1128}
the cases
$n =3$ and $n \geq 4$
are dissociated.
Proposition
\ref{19}
and Theorem \ref{19bis}, 
interesting in their own rights, 
play an important role in the 
characterization of the 2-Salem elements.

\begin{proposition}
\label{19}
Let $\Lambda$ be
the polynomial defined by
\begin{equation}
\label{22}
    \Lambda(Y)=Y^{3}+\lambda_{2}Y^{2}+\lambda_{1}Y
    +\lambda_{0}\in\mathbb{F}_{q}[X][Y]
\end{equation}
where  $2 \deg\lambda_{2}
<\displaystyle\deg\lambda_{1}=
\deg\lambda_0$. 
Suppose  $q\neq 2^r$ for any $r \geq 1$.
Let $\omega_1$ be a root of $\Lambda$ such that 
$|\omega_1|>1$. 
Then
  $\omega_1\in \mathbb{F}_q((X^{-1}))$ 
  if and only if $[\omega_1]\in \mathbb{F}_{q}[X]$
  and 
  $\deg\lambda_{1}$ is even
  ($\neq 0$).
\end{proposition}
\begin{proof}
The condition is necessary.
Indeed,
from
Theorem \ref{06},
the root $\omega_1$
belongs to $\overline{\mathbb{F}_q((X^{-1}))}$.
Imposing $\omega_1 \in
\mathbb{F}_q((X^{-1}))$
implies $[\omega_1]\in \mathbb{F}_{q}[X]$, and,
from Lemma \ref{07}, 
$\deg \lambda_{1}$ is even.
For sufficiency,
we consider that the decomposition 
$\omega_1 = [\omega_1] +1/Z_1 ,$
with $|Z_1| > 1$, holds, and
we keep the same notations
for $[\omega_1] $
as in \eqref{omega1omega2}.
Then the steps of the proof are those
of the proof of Lemma \ref{1211}
until the equality \eqref{13}.

In view of \eqref{141}, with
$\deg \lambda_1 = 2s > 0$,
we can write
\begin{align*}
  A_3 &=  [\omega_1]^{3}+\lambda_{2}[\omega_1]^{2}
  +\lambda_{1}[\omega_1]+\lambda_0 \\
     &= [\omega_1]^{3}-( [\omega_1]+[\omega_2]+[\omega_3])[\omega_1]^{2}+([\omega_1][\omega_2]+[\omega_1][\omega_3]+[\omega_2][\omega_3]+Q)[\omega_1]\\
&  \quad   -[\omega_1][\omega_2][\omega_3]+Q' \\
    & = Q"
\end{align*}
where $\deg Q \leq s-1$,
and $Q'$ and $Q"$ are two polynomials with degree 
less than or equal to $2s-1$.
Thus
\begin{equation*}
    \deg A_3\leq 2s-1.
\end{equation*}
Notice that $\displaystyle\deg\lambda_{1}> 2 \deg\lambda_{2}$ implies
\begin{equation}\label{13bis}
    \deg \lambda_{2} = \deg([\omega_1]+[\omega_2]+[\omega_3])<s.
\end{equation}
then $a_s + b_s = 0$. Hence
$$[\omega_1] - [\omega_2] = 2a_sX^s + (a_{s-1} - b_{s-1})X^{s-1} +\ldots+ (a_0 - b_0).$$
Since   $q\neq 2^r$, then $ \deg([\omega_1] - [\omega_2]) = s$.
Since
$$A_2 = 3 [\omega_1]^2
+ 2 \lambda_2 [\omega_1]+
\lambda_1
=
([\omega_1] - [\omega_2]+
[\omega_3]) [\omega_1] + 
[\omega_2][\omega_3] + Q$$
we have
$$\deg A_{2}=\deg([\omega_1]-[\omega_2])+s=2s .$$
We have
$$\deg A_{1}=s.$$
Notice that $A_3\neq0$; if not, by Corollary \ref{04}, we 
would have $|Z_1|\leq1$, a contradiction.
We conclude that
$$
   \deg A_{2}>\displaystyle\max_{i\neq 2}\deg A_i.
$$
Finally, by Proposition \ref{03},
 the only root of $H$ with an absolute value 
 $>1$ is $Z_1$ and  
$H$ admits  
the factor $(Z-Z_1)\in\mathbb{F}_q((X^{-1}))[Z]$.
Then $Z_1\in\mathbb{F}_{q}((X^{-1}))$ and 
$\omega_1=[\omega_1]+\displaystyle\frac{1}{Z_1} 
\in \mathbb{F}_{q}((X^{-1}))$, 
completing the proof.
\end{proof}

\begin{theorem}
\label{19bis}
Let $n\geq4$ and suppose $q \neq 2^r$ for any $r \geq 1$. 
Let $\Lambda$ be the polynomial   
\begin{equation}\label{22_}
    \Lambda(Y):=Y^{n}+\lambda_{n-1}Y^{n-1}+
    \lambda_{n-2}Y^{n-2}+\ldots+\lambda_{1}Y+\lambda_{0}
    \quad\in\mathbb{F}_{q}[X][Y]
\end{equation}
with $\lambda_0 \neq 0$.
Suppose  
$
\displaystyle\deg\lambda_{n-2}
=\max_{i\neq n-2}\deg(\lambda_i)
~~ {\rm and}
~~
 \deg\lambda_{n-2}>2\deg\lambda_{n-1}.$ 
Let $\omega_1$ be a root of $\Lambda$ 
such that $|\omega_1|>1$. Then
  $\omega_1\in \mathbb{F}_q((X^{-1}))$ 
  if and only if 
  $[\omega_1]\in \mathbb{F}_{q}[X]$,  $\deg\lambda_{n-2}$ is even ($\neq 0$)
  and
 $\deg\lambda_{n-3}<\deg\lambda_{n-2}$.
\end{theorem}
\begin{proof}
Let us show that the condition is necessary.
From
Theorem \ref{06}
the root $\omega_1$
belongs to $\overline{\mathbb{F}_q((X^{-1}))}$.
Assuming $\omega_1 \in
\mathbb{F}_q((X^{-1}))$
implies $[\omega_1]\in \mathbb{F}_{q}[X]$;
from Lemma \ref{07}, 
$\deg \lambda_{n-2}$ is even,
and, from 
Lemma \ref{1211}, 
$\deg\lambda_{n-3}<\deg\lambda_{n-2}$.

For sufficiency,
we consider that the 
root 
$\omega_1 \in
\overline{\mathbb{F}_{q}((X^{-1}))}
$
can be
decomposed as 
$\omega_1 = [\omega_1] +1/Z_1 ,$
with $|Z_1| > 1$ and
$[\omega_1] \in \mathbb{F}_{q}[X]$.
We keep the same notations
for $[\omega_1] $
as in \eqref{omega1omega2}.
The steps of the proof are now those
of the proof of Lemma \ref{1211}
until the equality \eqref{13}.
Denote  $2 s:= \deg\lambda_{n-2} > 0$.
We have 
$\deg \lambda_{n-3}\leq 2s-1$.

Since $n \geq 4$, 
the assumption $
\displaystyle\deg\lambda_{n-2}
=\max_{i\neq n-2}\deg(\lambda_i)$
means that the upper Newton polygon of
$\Lambda$ has an horizontal facet of length
$\geq 2$. Then there exists at least one root
of $\Lambda$,
say $\omega_3$, such that
$|\omega_3|=1$.
Using the expressions of the
symmetric functions $\lambda_j$s  of the roots
$\omega_1, \omega_2, \omega_3, \ldots$
as functions of
$[\omega_1], [\omega_2], [\omega_3], \ldots$, as
above, in \eqref{141}, i.e. in
$$
  A_n =  [\omega_1]^{n}+\lambda_{n-1}[\omega_1]^{n-1}+
  \lambda_{n-2}[\omega_1]^{n-2}+
  \ldots+\lambda_0,
$$
we deduce
$\deg A_n \leq (n-1)s-1$.

From the assumption
$2s= \displaystyle\deg\lambda_{n-2}> 2 \deg\lambda_{n-1}$
we deduce
\begin{equation}
\label{13ter}
 \deg \lambda_{n-1} = \deg([\omega_1]+[\omega_2]+[\omega_3]+
\sum_{j=4}^{n} [\omega_j])<s.
\end{equation}
Hence $a_s + b_s = 0$.
The condition  $q\neq 2^r$,
$r \geq 1$, implies
  $a_s\neq b_s$ and  $[\omega_1]\neq[\omega_2]$.
Hence the degree of
$$ [\omega_1]-[\omega_2]=2a_sX^s+(a_{s-1}-b_{s-1})X^{s-1}
+\ldots+(a_0-b_0)$$
is exactly $\deg([\omega_1]-[\omega_2])=s$.

Now the expressions of the coefficients
$A_{n-1}$ and $A_{n-2}$ are respectively:
\begin{eqnarray*}
  A_{n-1} &=& [\omega_1]^{n-2}([\omega_1]-[\omega_2])+(n-2)(Q[\omega_1]^{n-3}+\lambda_{n-3}[\omega_1]^{n-4}) \\
   &&-\lambda_{n-3}[\omega_1]^{n-4}+(n-4)\lambda_{n-4}[\omega_1]^{n-5}+\ldots+\lambda_1
\end{eqnarray*}
and
\begin{eqnarray*}
 A_{n-2} &=& (n-1)[\omega_1]^{n-3}([\omega_1]-[\omega_2]) +[\omega_1]^{n-3}[\omega_2]+ \\
 &+& \frac{(n-2)(n-3)}{2}[\omega_1]^{n-4}Q+\frac{(n-3)(n-4)}{2}\lambda_{n-3}[\omega_1]^{n-5}+\ldots+\lambda_1.
\end{eqnarray*}
Therefore
$$\deg A_{n-1}=(n-2)s+\deg([\omega_1]-[\omega_2])=(n-1)s$$
 and
$$\deg A_{n-2}=(n-2)s.$$
We have:
$\deg A_{n} < \deg A_{n-1},
\deg A_{n-2} < \deg A_{n-1}$
and it is easy to show
$$
   \displaystyle\max_{i\neq n-1}\deg A_i
   < \deg A_{n-1}.
$$
Now $A_n\neq0$; if not, 
by Corollary \ref{04}, we would have 
$|Z_1|\leq1$, a contradiction.
Finally, by Proposition \ref{03},
 the only root of $H$ which has 
 an absolute value~$>1$ 
 is $Z_1$ and $H$ admits the factor 
 $(Z-Z_1)\in\mathbb{F}_q((X^{-1}))[Z]$.
Then $Z_1\in\mathbb{F}_{q}((X^{-1}))$ and 
$\omega_1=[\omega_1]+\displaystyle\frac{1}{Z_1} 
\in \mathbb{F}_{q}((X^{-1}))$, completing the proof.
\end{proof}
\begin{remark}
 \end{remark}
\begin{itemize}
  \item[(i)]We mention that Theorem \ref{19bis} is not 
  always true in characteristic $3$ in the case
  \textbf{$\displaystyle\deg\lambda_{n-2}
  = 2 \deg\lambda_{n-1}$}  (see Example \ref{23}).
\item[(ii)] We note also that this theorem  is not 
always  true for any field of characteristic $p = 2$
 (see Example \ref{24}).
\end{itemize}
\begin{example}\label{23}
\end{example}
Let
\begin{equation}\label{25}
\Lambda(Y)=Y^3+(X+1)Y^2+X^2Y-X^2+2\in \mathbb{F}_3[X][Y].
\end{equation}
By Theorem \ref{06},
$\Lambda(Y)$ has two roots  $\omega_1$ and $\omega_2$ 
having an absolute value  strictly greater than $1$ 
and one root $\omega_3$
having an absolute value equal to $1$.
Set $\omega_1= X+ \displaystyle\frac{1}{Z_1}\in\mathbb{F}_{3}((X^{-1}))$ such that $|Z_1|>1$.
$Z_1$ is the root of  the polynomial defined by
\begin{equation}\label{eqI}
    2Z^3+2XZ^2+(X+1) Z+ 1=0.
\end{equation}
By Proposition \ref{03}, we deduce that $Z_1\in\mathbb{F}_{4}((X^{-1}))$
and  $\omega_1\in\mathbb{F}_{3}((X^{-1}))$.\\
Now set $\omega_2= X+1+ \displaystyle\frac{1}{Z_2}\in\mathbb{F}_{3}((X^{-1}))$ with $|Z_2|>1$.
 We obtain
$Z_2$ as a root of  the polynomial defined by
\begin{equation}\label{eqI_}
    Z^3+(X^2+X+1)Z^2+(2X^2+X+2) Z+ 1=0.
\end{equation}
Again by Proposition \ref{03}, we deduce that $Z_2\in\mathbb{F}_{3}((X^{-1}))$
and  $\omega_2\in\mathbb{F}_{3}((X^{-1}))$.\\
Since $\Lambda$ is monic and irreducible over $\mathbb{F}_3[X]$,
it follows that $(\omega_1,\omega_2)$ is 
a 2-Salem series  and $\Lambda$ is the 
minimal polynomial of $\omega_1$.
\vspace{0.4cm}

\noindent
{\it Proof of Theorem \ref{1128}.}
Let us prove the necessary condition for
(i) and (ii).
Assume that  $\omega_1\in \mathbb{F}_q((X^{-1}))$
and
$n \geq 3$. 
By Proposition
\ref{19}
or
Theorem \ref{19bis}, and the notations
in their respective proofs, we deduce that
$\deg\lambda_{n-2}$ is even
and $\neq 0$. 
Still with these notations, set
 \begin{align}
 \label{28}
\lambda_{n-2}&=\alpha_{2s}X^{2s}+\alpha_{2s-1}X^{2s-1}+\ldots+\alpha_{0}
=
[\omega_1][\omega_2]+Q
\\
\label{28bis}
 &= (a_{s}X^{s}+ a_{s-1}X^{s-1}+\ldots+a_0)(b_{s}X^{s}+ b_{s-1}X^{s-1}+\ldots+b_0)+Q.
\end{align} 
From \eqref{13bis}
or \eqref{13ter},
we have $\deg\lambda_{n-1}<s$.
Hence
    $a_s = -b_{s} \in \mathbb{F}_q$, 
    what implies the claim
$$  -\alpha_{2 s} = -a_{s} b_{s} = a_{s}^{2}
\neq 0.
$$
In addition, for $n \geq 4$,
Theorem \ref{19bis} implies that
$\deg\lambda_{n-3}<\deg\lambda_{n-2}$ holds.

Let us prove the sufficient condition for (i).
By Theorem \ref{06} the polynomial
$\Lambda$ has two roots
$\omega_1$ and
$\omega_2$ such that
$|\omega_1|>1, 
|\omega_2|>1$, with
at least one
conjugate $\omega_j$, $3\leq j\leq n$, 
on the unit circle.
Let $k$ denote the length
of the horizontal facet
of the upper Newton polygon.
Since $\deg\lambda_{n-3}<\deg\lambda_{n-2}$,
we have $k \geq 2$.
There are $k$ conjugates $\omega_{j}$,
$j=3,\ldots, 3+k-1$,
on the unit circle, by
Proposition \ref{03}.
Let
$$\omega_{j}
= c_{0}^{(j)} + c_{-1}^{(j)} X^{-1} 
+ \ldots \in \overline{\mathbb{F}_{q}((X^{-1}))},\qquad
j=3, \ldots, 3+k-1.$$
From Proposition \ref{03} (ii), 
we can see
$\displaystyle\sum_{j=3}^{3+k-1}
\omega_{j}
\in \mathbb{F}_{q}((X^{-1}))$
and therefore
$\displaystyle\sum_{i=3}^{3+k-1}
c_{0}^{(i)} \in \mathbb{F}_q$. 
Now
$$\lambda_{n-1} 
= 
\beta_s X^s + \beta_{s-1} X^{s-1} + \ldots 
+ \beta_0$$
$$
= -\bigl([\omega_1 ] + [\omega_2 ] +
\sum_{i=3}^{3+k-1}
c_{0}^{(i)}
\bigr).$$
Thus
\begin{equation}
\label{betta0i}
-\beta_i = a_i + b_i, \qquad
1 \leq i \leq s.
\end{equation}
and
\begin{equation}
\label{betta00}
-\beta_0 = a_0 + b_0 +
\sum_{i=3}^{3+k-1}
c_{0}^{(i)} .
\end{equation}

Suppose  $\alpha_{2s}= -a^2$
where $s \geq 1$ and
$a \in \mathbb{F}_q$ is nonzero. 
Let us put $a_s=a$. Then $b_s = -a$
and $\beta_s = 0$.
We deduce
$$\alpha_{2s -1}
=a_{s}b_{s-1}+a_{s-1}b_{s}
=a(b_{s-1}-a_{s-1}),
$$
then $b_{s-1}-a_{s-1} \in \mathbb{F}_q$.
Since
$q \neq 2^r$, for any $r \geq 1$,
and that
$b_{s-1}+a_{s-1} = -\beta_{s-1} \in \mathbb{F}_q$,
we have
$$a_{s-1}, b_{s-1} \in \mathbb{F}_q .$$
Let us show recursively that
$$a_{s-i}, b_{s-i} \in \mathbb{F}_q , \qquad i=2, 3, \ldots, s.$$
Let us assume that
$a_{s-j}, b_{s-j} \in \mathbb{F}_q $ holds
for $j=0,1, \ldots, i-1$.
 From \eqref{28bis}, we deduce 
\begin{align*}\label{as}
    \alpha_{2s-i}&=a_{s}b_{s-i}+
    a_{s-1}b_{s-i+1}+\ldots+a_{s-i}b_s\\
                &=a(b_{s-i}-a_{s-i})+d_{s-i}
\end{align*}
where
$$
   d_{s-i}:=a_{s-1}b_{s-i+1}+\ldots +b_{s-1}a_{s-i+1}
   \in \mathbb{F}_q , \qquad
   i=2, \ldots, s.
$$
Hence
\begin{equation}
\label{diff}
b_{s-i}-a_{s-i} = a^{-1} (\alpha_{2s-i}
-d_{s-i})
\in \mathbb{F}_q .
\end{equation}
Since
$b_{s-i}+a_{s-i} = -\beta_{s-i} \in \mathbb{F}_q$,
we have
$$a_{s-i}, b_{s-i} \in \mathbb{F}_q .$$

Let us note  $d_{s-1}=0$.
Combining \eqref{betta0i}
\eqref{betta00} and \eqref{diff}, we obtain
\begin{equation}
\label{result_ai}
a_i = - 2^{-1}(\beta_{i} + 
a^{-1}( \alpha_{s+i} - d_i)),
\qquad 0\leq i \leq s-1.
\end{equation}
Therefore, $[\omega_1]\in \mathbb{F}_{q}[X]$ 
and from Theorem \ref{19bis},  we obtain $\omega_1\in\mathbb{F}_q((X^{-1}))$.
In the same way, we can show 
that $\omega_2\in\mathbb{F}_q((X^{-1}))$.
As $\Lambda$ is monic and irreducible over 
$\mathbb{F}_q[X]$, then $\omega_1$ is an algebraic integer.
Therefore $(\omega_1,\omega_2)$ is a 2-Salem element
in $T'^{\ast}_{2}$.
\vspace{0.2cm}

Let us give the proof of the sufficency condition
for (ii), in the same way.
By Theorem \ref{06} the polynomial
$\Lambda$ has two roots
$\omega_1$ and
$\omega_2$ such that
$|\omega_1|>1, 
|\omega_2|>1$, and the third one
$\omega_3$ is
on the unit circle.
For $n=3$,
the assumptions
$\deg\lambda_{1}
=\deg(\lambda_0)$ 
and
$\deg\lambda_{1}
> 2\deg(\lambda_2)$ hold.
Then
Proposition \ref{19} can be applied to
obtain the result.
We have just to show that
$[\omega_1]\in \mathbb{F}_{q}[X]$.
For proving 
$[\omega_1]\in \mathbb{F}_{q}[X]$
we proceed as above, from
\eqref{28} to \eqref{result_ai},
except
that $-\beta_0$ is now equal to
$a_0 + b_0 +
c_{0}$ with
$$\omega_3 = c_0 + c_1 X^{-1}+\ldots
\quad \in \overline{\mathbb{F}_{q}((X^{-1}))} .$$

(iii) This assertion follows immediately from 
Corollary \ref{05}.
\begin{flushright}
$\Box$
\end{flushright}
\begin{remark}
\end{remark}
Note that Theorem \ref{1128} (i) is not always true in the case $\deg\lambda_{n-2}=2\deg\lambda_{n-1}$.
To show this, we construct two counter-examples.
\begin{example}
\end{example}
Let $\Lambda$ the polynomial over $\mathbb{F}_3[X]$ 
which is  defined by \eqref{25}. 
Then,
in view of the above,   $\Lambda$ satisfies
the conditions $\deg\lambda_{n-2}=2\deg\lambda_{n-1}$ 
and $-1$ is not a square in $\mathbb{F}_3$.
In contrast, $\Lambda$ has two dominant roots 
$\omega_1, \omega_2\in\mathbb{F}_{3}((X^{-1}))$.
\begin{example}
\end{example}
The polynomial
$$
    \Lambda_2=Y^4-XY^3+X^2Y^2+XY+X^2+1 \in\mathbb{F}_5[X][Y]
$$
satisfies the conditions $\deg\lambda_{n-2}=2\deg\lambda_{n-1}$ and $-1$ is a square in $\mathbb{F}_5$.
By Proposition \ref{03} $(i)$,  $\Lambda_2$ has exactly two dominant roots  $ \omega_1$  and $ \omega_2$ with
$$
\deg\omega_1=\deg\omega_2=1.
$$
The other conjugated roots  $ \omega_3$  and $ \omega_4$ have the same degree equal to $0$.
Suppose $[\omega_1]\in \mathbb{F}_{5}[X]$, using the fact that
$$[\omega_1]+ [\omega_2]+[\omega_3]+ [\omega_4]=X,
$$ this yields that $[\omega_2]\in \mathbb{F}_{5}[X]$.
Let
$$  [\omega_1]= a_{1}X+ a_{0}~~ ,\quad~~  [\omega_2]= b_{1}X+ b_{0} $$
 and
$$  [\omega_3]= c_{0}~~ ,\quad~~  [\omega_2]=  d_{0} $$
where $a_{1}, b_{1},  c_{0}$ and $d_{0}$ are four integers in $\mathbb{F}_5\backslash\{0\}$. It follows that
$$
   a_{1}+ b_{1}=  a_{1} b_{1}=1.
$$

These equations have no solutions in $\mathbb{F}_{5}$.

\section{A criterium of irreducibility}

In the following the assumption
``$\lambda_0 \neq 0$" is replaced by
the stronger hypothesis
``$\Lambda$ has no root in $\mathbb{F}_q$"
in order to reach the property of being
irreducible.

\begin{lemma} \label{08}
Let $n \geq 3$.
Let $\Lambda$ be defined by
$$\Lambda(Y)=Y^{n}+\lambda_{n-1}Y^{n-1}+\lambda_{n-2}Y^{n-2}
+\ldots+\lambda_{1}Y+\lambda_{0}
\quad\in\mathbb{F}_{q}[X][Y].$$
Suppose  that $\Lambda$ has no root
in $\mathbb{F}_q$ and 
$\displaystyle\max_{i < n-3}\deg\lambda_i
< 
\deg\lambda_{n-3}=
\deg\lambda_{n-2}\geq2\deg\lambda_{n-1}$.
If  $\deg(\lambda_{n-2})$ is odd, then $\Lambda$ is irreducible over $\mathbb{F}_q[X]$.
\end{lemma}

\begin{proof}
By considering the upper Newton polygon of $\Lambda$, 
the polynomial $\Lambda$ has exactly two roots 
$\omega_1$ and $\omega_2$ such that  
$|\omega_1| > 1$ and $|\omega_2|>1$, one
root $\omega_3$ such that $|\omega_3|=1$ and
the
remaining roots $\omega_4,\ldots,\omega_n$ have an 
absolute value strictly less than $1$. 
Suppose that $\Lambda(Y)$ admits the 
decomposition
$$\Lambda(Y)~=~\Lambda_1(Y).\Lambda_2(Y)$$
\begin{equation} \label{09}
=~ (Y^s+A_{s-1}Y^{s-1}+\ldots+A_{1}Y+A_0)
(Y^m+B_{m-1}Y^{m-1}+\ldots+B_{1}Y+B_0)
\end{equation}
with $\Lambda_1, \Lambda_2\in\mathbb{F}_{q}[X][Y]$ 
and  $s>0,~m>0$.

There are several cases to show the contradiction. 
If we had
$\Lambda_1(\omega_i)=0 $ for 
$i = 1, 2, 3$, all the
roots of $\Lambda_2$ would have an absolute value 
strictly less  $1$,  which is a contradiction, because
 $|B_0|>1$. 
If we had 
$\Lambda_1(\omega_i)=0 $ for 
$i = 1, 2$, and $\Lambda_{2}(\omega_3) =0$,
with $m = \deg\Lambda_2 > 1$, 
then one of the 
roots of $\Lambda_2$ would have an
absolute value equal to $1$ and the
other roots of $\Lambda_2$
have an absolute value strictly less  $1$,  
which is a contradiction,
since $|B_0|>1$. 
Now, if $\Lambda_1(\omega_1)=\Lambda_1(\omega_2)=0$,
and $\Lambda_2(\omega_3)=0$ with 
$\deg\Lambda_2 = 1$,
all the other conjugates of $\omega_1$ 
are roots of  $\Lambda_1$,  
then, from
\eqref{09},
\begin{equation}
\label{decomp}
    \Lambda(Y) =(Y^{n-1}+A_{n-2}Y^{n-2}+\ldots+A_{1}Y
    +A_0)(Y+B_0)
    \quad \in \mathbb{F}_q[X][Y]
\end{equation}
with $\deg B_0=\deg (\omega_3)=0$, and then
$B_0=b_0\in \mathbb{F}_q\backslash\{0\}$. 
This is in contradiction with the assumption.

Then we can conclude that 
$\Lambda_1(\omega_1)=0 $ and $\Lambda_2(\omega_2)=0$.
The remaining roots of $\Lambda_1$ and $\Lambda_2$ 
have an absolute value $\leq1$. 
\qquad \qquad$(\ast\ast\ast)$

Let us continue the generic case, assuming
$\Lambda_{1}(\omega_1)=
\Lambda_{1}(\omega_3)
=0$
and  $\Lambda_{2}(\omega_2)=0$.
Since $-A_{s-1}$ (resp. $-B_{m-1}$) is the sum 
of the roots of $\Lambda_1$ (resp. $\Lambda_2$) 
and by the symmetric functions of the roots, 
it follows that
$$\deg A_{s-1}
=\deg \omega_1
= \displaystyle\max_{i\neq s-1}\deg A_{i} 
\quad{\rm and}\quad 
\deg B_{m-1}
=\deg\omega_2
>\displaystyle\max_{j\neq m-1}\deg B_{j}.$$
In particular we have:
$|A_{s-2}|\leq |\omega_1|$ and $|B_{m-2}|<|\omega_2|$.
Then
   $$\deg\lambda_{n-2}= 
   \deg (A_{s-2}+ A_{s-1} B_{m-1}+ B_{m-2})
   =\deg A_{s-1}+\deg B_{m-1} .$$
But the assumption
$\deg\lambda_{n-2}\geq
2\deg\lambda_{n-1}$ means that
$$\deg A_{s-1}+\deg B_{m-1} \geq
2 \max\{\deg A_{s-1}, \deg B_{m-1}\},$$
from which we deduce
$$\deg A_{s-1}= \deg B_{m-1},$$
and then $\deg\lambda_{n-2}= 2 \deg A_{s-1}$.
By Lemma \ref{07}, 
$\Lambda$ would have no root
in  $\mathbb{F}_q((X^{-1}))$ 
with absolute value $ > 1$,
 a contradiction.
 We deduce the irreducibility
 of $\Lambda$ over $\mathbb{F}_{q}[X]$.
\end{proof}

\begin{theorem}\label{31}

Let $n\geq4$ and suppose $q \neq 2^r$ for any $r \geq 1$. 
Let $\Lambda$ be the polynomial   
\begin{equation}\label{202}
    \Lambda(Y):=Y^{n}+\lambda_{n-1}Y^{n-1}+
    \lambda_{n-2}Y^{n-2}+\ldots+\lambda_{1}Y+\lambda_{0}
    \quad\in\mathbb{F}_{q}[X][Y].
\end{equation}
Suppose  that
$\Lambda$ has no root in
$\mathbb{F}_q$,
and assume that
the coefficients $\lambda_i$ satisfy
\begin{itemize}
\item[(i)]
$
\displaystyle
\max_{i\in\{1,2,\ldots,n-4\}\cup \{n-1\}}
\deg\lambda_i < 
\deg\lambda_{n-3} 
=\deg\lambda_{n-2} < 2\deg \lambda_{n-1}$,
\item[(ii)]
$\displaystyle
\frac{\deg \lambda_{i+1}+\deg\lambda_{i-1}}{2}<\deg\lambda_{i}\qquad {\rm for} ~~1\leq i \leq n-4$, 
\item[(iii)]
$\displaystyle\deg\lambda_{n-2}-\deg\lambda_{n-1}<\deg\lambda_{n-4}<\deg\lambda_{n-1}$.
\end{itemize}
\vspace{0.1cm}

\noindent
Then
$(\omega_1,\omega_2)$ is a  \textit{2-Salem} element 
and $\Lambda$ is its minimal polynomial.
 \end{theorem}
 \begin{proof}
By Corollary \ref{05}, $\Lambda(Y)$ 
has two roots 
$\omega_1$ and $\omega_2$ 
in  $\mathbb{F}_q((X^{-1}))$,
such that
$|\omega_1|>1$ and $|\omega_2|>1$,
and there is exactly one conjugate
$\omega_3$ which
lies on the unit circle.
Denote $s := \deg (\omega_2 )$ and 
$m:= \deg (\omega_1 )$ respectively.
They satisfy
$$1<|\omega_2|= q^{\displaystyle\deg\lambda_{n-2}-\deg\lambda_{n-1}} =q^m< |\omega_1|= q^{\displaystyle\deg\lambda_{n-1}}=q^s .$$
The other conjugates $\omega_4,\ldots, \omega_n
\in \overline{\mathbb{F}_q((X^{-1}))}$
have an absolute value strictly less than $1$. 
Since  
$\displaystyle\frac{\deg \lambda_{i+1}+
\deg\lambda_{i-1}}{2}<\deg\lambda_{i}$,
for $1\leq i \leq n-4$, then
$\deg \lambda_{i+1}
-\deg\lambda_{i}
<
\deg\lambda_{i}
-\deg\lambda_{i-1}$;
all the facets of the upper Newton polygon of $\Lambda$ are of length 1.
We have
$$|\omega_j| = q^{-k_j}<1, \qquad \quad
4\leq j \leq n,$$
with
\begin{equation}
\label{k_j}
-k_j=\deg \omega_j= \deg \lambda_{n-j}
-\deg\lambda_{n-j+1}.
\end{equation}

We now assume that $\Lambda$
is reducible and
show the contradiction.
With the same notations 
as in the proof of
Lemma \ref{08}, let us suppose that 
$\Lambda(Y)$ admits the 
decomposition
$$\Lambda(Y)~=~\Lambda_1(Y).\Lambda_2(Y)$$
\begin{equation} \label{09_2}
=~ (Y^s+A_{s-1}Y^{s-1}+\ldots+A_{1}Y+A_0)
(Y^m+B_{m-1}Y^{m-1}+\ldots+B_{1}Y+B_0)
\end{equation}
as in
\eqref{09}.
Then we discard the impossible cases
as in the proof of 
Lemma \ref{08}, i.e. from
\eqref{09}
 until $(\ast\ast\ast)$ in the same steps.
We conclude that $\Lambda_1(\omega_1)=0 $ and 
$\Lambda_2(\omega_2)=0$.

Now suppose that $\Lambda_2(\omega_3)=0$, without loss of generality; 
so we obtain  $\omega_1 \in S^\ast$ and 
$\omega_2\in T^\ast$.
Applying Theorem \ref{01}, we get
\begin{equation}\label{12}
  s = \deg A_{s-1}=\deg \omega_1>
  \displaystyle\max_{i\leq s-2}\deg A_{i} 
  \end{equation}
  and
\begin{equation}\label{12_B} 
m = \deg B_{m-1}= \deg B_{m-2}
 =\deg\omega_2>\displaystyle\max_{j\leq m-3}\deg B_{j}.
\end{equation}

The contradiction will come
from the coefficient
$\lambda_{n-4}$.
From \eqref{09_2}, 
\begin{equation}\label{3}
\lambda_{n-4} = A_{s-4}+A_{s-3}B_{m-1}+A_{s-2}B_{m-2}+A_{s-1}B_{m-3}+ B_{m-4}  .
\end{equation}
Let us examine the degrees of the terms of the sum.
First we can see that $\Lambda_1(\omega_4)=0$. 
Indeed, if we assume
$\Lambda_1(\omega_4)\neq0$,
by the symmetric functions of the roots of  
$\Lambda_2$ we would obtain, using (i) and
\eqref{k_j},
\begin{equation*}
    \deg B_{m-3}=\deg (\omega_2\omega_3\omega_4)=\deg\lambda_{n-4}-\deg\lambda_{n-1}<0,
\end{equation*}
a contradiction.
In the list $\{\omega_1 ,
\omega_2 , \omega_3 ,
\omega_4 , \ldots , \omega_n\}$
the roots $\omega_1$ and
$\omega_4$ are roots of $\Lambda_1$,
the roots $\omega_2$ and
$\omega_3$ are roots of $\Lambda_2$,
and the other roots are distributed
as roots of $\Lambda_1$
or $\Lambda_2$. Then
$\deg  B_{m-3}
> 1$.
From \eqref{12} we deduce
\begin{equation}
\label{term13}
\max\{\deg A_{s-3}, \deg A_{s-4}\}
<
s = \deg A_{s-1}< \deg A_{s-1}+\deg  B_{m-3}
=
\deg (A_{s-1}B_{m-3}).
\end{equation}
On the other hand, 
$\deg A_{s-2} > 0$.
From \eqref{12_B} we deduce
$$\max\{\deg B_{m-3}, \deg B_{m-4}\}
<
\deg  B_{m-2} =m=\deg B_{m-1} 
< \deg A_{s-2}+\deg B_{m-2}.$$
Let us show that 
$  \deg(A_{s-2}B_{m-2})  <  s.$

Indeed, from (iii),
$ \lambda_{n-4}< \deg\lambda_{n-1}=s$; then
$$\deg \lambda_{n-4}=\deg(\omega_1\omega_2\omega_3\omega_4)=\deg(\omega_1)+\deg(\omega_2)+\deg(\omega_4)=s+m+\deg(\omega_4)<s.
$$
Thus
$$\deg(\omega_4)<-\deg(\omega_2)=-m,
$$
what means
$$  \deg(A_{s-2}B_{m-2}) = \deg(\omega_1)+\deg(\omega_4) +\deg B_{m-1} <  s -m+m=s.$$  
In the same way, using (ii),
$$  \deg(A_{s-3}B_{m-1}) = \deg(\omega_1)+\deg(\omega_4) +\deg(\omega_5)+\deg B_{m-1} <  s -m-(m+1)+m<s.$$

We deduce
$$\deg (\lambda_{n-4})=  
\deg (A_{s-1}B_{m-3}).$$
But, from \eqref{term13}, we have
$$\deg (\lambda_{n-4})=  
\deg (A_{s-1}B_{m-3}) > s.$$
The contradiction comes from
(iii) since $\deg (\lambda_{n-4})$ should be 
$< s = 
\deg \lambda_{n-1}$.

Therefore $\Lambda(Y)$ is irreducible 
over $\mathbb{F}_{q}[X]$.
Finally, since $\Lambda(Y) $ is monic, then 
$ (\omega_1, \omega_2) $ is a  2-Salem element
and $ \Lambda $ is its minimal polynomial.
\end{proof}

\begin{example}{ 2-Salem series of degree $5$ in $\mathbb{F}_3((X^{-1}))$.}
\end{example}
Let
$$\Lambda(Y) = Y^{5}+X^4 Y^4
+ X^{5} Y^3 + X^{5} Y^2 + X^3 Y + 1
\in \mathbb{F}_{3}[X][Y].$$
We deduce from Theorem \ref{31} that  
$\Lambda$ is irreducible over $\mathbb{F}_3[X]$ 
and has $5$  roots
defined by 
$$\left\{
\begin{array}{ll}
\omega_1= & \hbox{$X^5 +2X+\displaystyle\frac{1}{X^2}+\ldots=X^5 +2X+\displaystyle\frac{1}{Z_1}$~~~~~~ such that $|Z_1|>1$} \\
    \omega_2= & \hbox{$X+1+ \displaystyle\frac{1}{Z_2}$~~~~~~ ~~~~~~~~~~~~~~~~~~~~~~~~~~~~~~~~~~~such that $|Z_2|>1$} \\
     \omega_3= & \hbox{$2+ \displaystyle\frac{1}{Z_3}$ ~~~~~~~~~~~~~~~~~~~~~~~~~~~~~~~~~~~~~~~~~~~~~ ~~such that $|Z_3|>1$} \\
      \omega_4= & \hbox{$\displaystyle\frac{1}{X^2}+\ldots$ }\\
     \omega_5= & \hbox{$\displaystyle\frac{2}{X^3} +\ldots$}\\
  \end{array}
\right.$$
These roots correspond to the facets of the upper 
Newton polygon associated with the 
2-Salem minimal polynomial $\Lambda$.  
Since $\Lambda$ is monic then $w_1$ is an 
algebraic integer. 
Therefore  $(\omega_1,\omega_2)$
is a 2-Salem  element.

\end{document}